\documentclass[12pt]{amsart}
\usepackage{amscd,amsmath,amsthm,amssymb,graphics}
\usepackage{pstcol,pst-plot,pst-3d}%\usepackage[T1]{fontenc}
\usepackage{lmodern,pst-node}%\usepackage{geometric}
\usepackage{multicol}
\usepackage{epic,eepic}
\usepackage{amsfonts,amssymb,amscd,amsmath,enumerate,verbatim}
\psset{unit=0.7cm,linewidth=0.8pt,arrowsize=2.5pt 4}
\newpsstyle{fatline}{linewidth=1.5pt}
\newpsstyle{fyp}{fillstyle=solid,fillcolor=verylight}
\definecolor{verylight}{gray}{0.97}
\definecolor{light}{gray}{0.9}
\definecolor{medium}{gray}{0.85}
\definecolor{dark}{gray}{0.6}

\usepackage{etex}
\usepackage{tikz}

\newcommand{\vertex}{\node[vertex]}
\tikzstyle{vertex}=[circle, draw, inner sep=0pt, minimum size=4pt]
\tikzset{main node/.style={circle,fill=blue!15,draw,minimum size=.5cm,inner sep=0pt},
}
\usepackage{tikz}
\usetikzlibrary{positioning}
\usepackage{tikz,fullpage}
\usetikzlibrary{arrows,%
                petri,%
                topaths, automata}%
\usepackage[position=top]{subfig}

\tikzstyle{vertex}=[circle, draw, inner sep=0pt, minimum size=6pt]
\def\frk{\frak}               % font for "Fraktur"

\def\Phi{{\frk n}}
\def\Phi{{\frk N}}
\def\opn#1#2{\def#1{\operatorname{#2}}} % to make operators
\opn\chara{char} \opn\length{\ell} \opn\pd{pd} \opn\rk{rk}
\opn\projdim{proj\,dim} \opn\injdim{inj\,dim} \opn\rank{rank}
\opn\depth{depth} \opn\grade{grade} \opn\hte{ht}
\opn\embdim{emb\,dim} \opn\codim{codim}

\opn\Tr{Tr} \opn\bigrank{big\,rank}
\opn\superheight{superheight}\opn\lcm{lcm}
\opn\trdeg{tr\,deg}%\emph{
\opn\reg{reg} \opn\lreg{lreg} \opn\ini{in} \opn\lpd{lpd}
\opn\size{size}\opn\bigsize{bigsize}
\opn\cosize{cosize}\opn\bigcosize{bigcosize}
\opn\sdepth{sdepth}\opn\sreg{sreg}
\opn\link{link}\opn\fdepth{fdepth}
\opn\div{div} \opn\Div{Div} \opn\cl{cl} \opn\Cl{Cl}
\opn\Spec{Spec} \opn\Supp{Supp} \opn\supp{supp} \opn\Sing{Sing}
\opn\Ass{Ass} \opn\Min{Min}\opn\Mon{Mon} \opn\dstab{dstab} \opn\astab{astab}
\opn\Syz{Syz}
\opn\Ann{Ann} \opn\Rad{Rad} \opn\Soc{Soc}
\opn\Im{Im} \opn\Ker{Ker} \opn\Coker{Coker} \opn\Am{Am}
\opn\Hom{Hom} \opn\Tor{Tor} \opn\Ext{Ext} \opn\End{End}
\opn\Aut{Aut} \opn\id{id}

\opn\nat{nat}
\opn\pff{pf}%   \pf exists already
\opn\Pf{Pf} \opn\GL{GL} \opn\SL{SL} \opn\mod{mod} \opn\ord{ord}
\opn\Gin{Gin} \opn\Hilb{Hilb}\opn\sort{sort}
\opn\initial{init}
\opn\ende{end}
\opn\height{ht}
\opn\type{type}
\opn\aff{aff} \opn\con{conv} \opn\relint{relint} \opn\st{st}
\opn\lk{lk} \opn\cn{cn} \opn\core{core} \opn\vol{vol}
\opn\link{link} \opn\star{star}\opn\lex{lex}
\opn\gr{gr}

\def\pot#1#2{#1[\kern-0.28ex[#2]\kern-0.28ex]}

\opn\dirlim{\underrightarrow{\lim}}
\opn\inivlim{\underleftarrow{\lim}}
\def\Implies{\ifmmode\Longrightarrow \else
        \unskip${}\Longrightarrow{}$\ignorespaces\fi}
\def\implies{\ifmmode\Rightarrow \else
        \unskip${}\Rightarrow{}$\ignorespaces\fi}
\def\iff{\ifmmode\Longleftrightarrow \else
        \unskip${}\Longleftrightarrow{}$\ignorespaces\fi}

\let\:=\colon
\newtheorem{Theorem}{Theorem}[section]
 \newtheorem{Lemma}[Theorem]{Lemma}
 \newtheorem{Corollary}[Theorem]{Corollary}
 \newtheorem{Proposition}[Theorem]{Proposition}

 \newtheorem{Example}[Theorem]{Example}

\let\epsilon\varepsilon
\let\kappa=\varkappa
\def\qed{\ifhmode\textqed\fi
      \ifmmode\ifinner\quad\qedsymbol\else\dispqed\fi\fi}
\def\textqed{\unskip\nobreak\penalty50
       \hskip2em\hbox{}\nobreak\hfil\qedsymbol
       \parfillskip=0pt \finalhyphendemerits=0}
\def\dispqed{\rlap{\qquad\qedsymbol}}

\opn\dis{dis}
\def\pnt{{\raise0.5mm\hbox{\large\bf.}}}

\opn\Lex{Lex}

\begin{document}

\author[Mafi and Naderi]{Amir Mafi and Dler Naderi}
\title{Almost Cohen-Macaulay bipartite graphs and connected in codimension two}

\address{Amir Mafi, Department of Mathematics, University Of Kurdistan, P.O. Box: 416, Sanandaj, Iran.}
\email{A\_Mafi@ipm.ir}
\address{Dler Naderi, Department of Mathematics, University of Kurdistan, P.O. Box: 416, Sanandaj,
Iran.}
\email{dler.naderi65@gmail.com}

\begin{abstract}
In this paper we study almost Cohen-Macaulay bipartite graphs. Furthermore, we prove that if $G$ is almost Cohen-Macaulay bipartite graph with at least one vertex of positive degree, then there is a vertex of $\deg(v) \leq 2$. In particular, if $G$ is an almost Cohen-Macaulay bipartite graph and $u$ is a vertex of degree one of $G$ and $v$ its adjacent vertex, then $G\setminus\{v\}$ is almost Cohen-Macaulay. Also, we show that an unmixed Ferrers graph is almost Cohen-Macaulay if and only if it is connected in codimension two. Moreover, we give some examples.
\end{abstract}


\maketitle
\section*{Introduction}
Throughout this paper, we assume that $G$ is a finite simple graph (without loops, multiplies edges and any isolated vertices) with vertex set $V (G)$ and edge set $E(G)$. For $W \subseteq V (G)$ we denote by $G \setminus W$ the subgraph of $G$ obtained by removing all
vertices of $W$ from $G$. Moreover, for any $v \in  V (G)$ we denote by $N_{G}(v)$ the neighbor
set of $v$ in $G$, i.e. $N_{G}(v) = \{u \in V (G) ~|~ \{u, v \} \in E(G) \}$. The inclusive
neighborhood of $v \in V(G)$ is the set $N_{G}[v]$ consisting of $v$ and vertices adjacent to $v$ in
$G$, i.e. $N_{G}[v]= N_{G}(v) \cup \{v \} $.

Let $R=k[x_1, \ldots, x_n]$ be the polynomial ring on $n$ variables over the  field $k$. We can associate to $G$ the ideal $I(G)$ of $R$ which is generated by all the square-free monomials $x_i x_j$ such that $x_i$ is adjacent to $x_j$. The ideal $I(G)$ is called the { \it edge ideal} of $G$. The complementary simplicial complex of $G$ is defined by
$ \Delta_{G} = \{ F \subseteq V(G) ~|~ F ~\textit{is an independent set in} ~G \}$, where $F$ is an independent set in $G$ if none of its elements are adjacent. Note that $\Delta_{G}$ is precisely the Stanley-Reisner simplicial complex of $I(G)$, i.e. $I_{\Delta_{G}}=I(G)$.

The graph $G$ is called Cohen-Macaulay (i.e. CM) if $R/I(G)$ is Cohen-Macaulay. Cohen-Macaulay graphs were studied in several works (see \cite{V2} and \cite{CRT}). A complete classification of Cohen-Macaulay graphs does
not exist. However, all Cohen-Macaulay bipartite graphs have been characterized in a combinatorial way by Herzog and Hibi in \cite{HH4}. A graph $G$ is called bipartite, if $V (G) = V_1 \cup V_2$ with $V_1 \cap V_2 =\emptyset$ such that
$E(G) \subseteq V_{1} \times V_{2}$.  It is easy to see that a graph $G$ is bipartite if and only if it has no cycle of odd length. For a
Cohen-Macaulay bipartite graph $G$, Estrada and Villarreal \cite{EV} showed that $G\setminus \{ u \}$ is Cohen-Macaulay for some vertex $u \in V(G)$. As usual $K_{m,n}$ will denote the complete bipartite
graph containing every edge joining $V_1$ and $V_2$, where $V_1$ and $V_2$ have $m$ and $n$ vertices
respectively and it is easy to see that $K_{m,n}$ is Cohen-Macaulay if and only if $m=n=1$.

A vertex cover of $G$ is a subset $C$ of $V (G)$ such that each edge has at least one vertex in $C$. A minimal vertex cover $C$ of $G$ is a vertex cover such that no proper subset of $C$ is a vertex cover of $G$. Observe that $C$ is a minimal vertex
cover if and only if $V(G) \setminus C$ is a maximal independent set. The graph $G$ is called unmixed if all its minimal vertex covers are of the same cardinality.  All unmixed bipartite graphs have been characterized by Villarreal in \cite{V1}.

Let $G$ be an unmixed bipartite graph with bipartition $V_1= \{ x_1, \ldots, x_n \}$ and $V_2= \{ y_1, \ldots, y_m \}$. Since $G$ is unmixed, it follows that $\hte(I)=n=m$. The edges
$ \{ x_i, y_i \}$ for  $i = 1, \ldots , n$ are called perfect matching edges of $G$. By \cite[Theorem 10.2]{Hara}, the height of $I(G)$ is equal to the
maximum number of independent lines in $G$, i.e. for any unmixed bipartite graph there is a perfect matching. Therefore we may assume that $\{ x_i, y_i \}$ is an edge of $G$ for all $i$. So each minimal vertex cover of $G$ is of the form $\{ x_{i_1} , \ldots , x_{i_s} , y_{i_{s+1}} , \ldots , y_{i_n} \}$, where $\{ i_1, \ldots , i_n \} = [n]$.

The graph $G$ is called almost Cohen-Macaulay (i.e. aCM) if $R/I(G)$ is aCM. We say that  $R/I(G)$ is aCM when $\depth  R/I(G)\geq \dim R/I(G) -1$.  The aCM modules has been studied in \cite{Ha, K1, K2, I, CTT, MT, TM, TMA, MN2}.

In this paper we study  almost Cohen-Macaulay bipartite graphs. Furthermore, we prove that if $G$ is almost Cohen-Macaulay bipartite graph with at least one vertex of positive degree, then there is a vertex of $\deg(v) \leq 2$. In particular, if $G$ is an almost Cohen-Macaulay bipartite graph and $u$ is a vertex of degree one of $G$ and $v$ its adjacent vertex, then $G\setminus\{v\}$ is almost Cohen-Macaulay. Also, we show that an unmixed Ferrers graph is almost Cohen-Macaulay if and only if it is connected in codimension two.
For any unexplained notion or terminology, we refer the reader to \cite{HH1} and \cite{V} . Several explicit examples were performed with help of the computer
algebra systems Macaulay2 \cite{G}.

\section{ Preliminary}
In this section, we recall some definitions and known results which is used in this paper. Let $\Delta$ be a simplicial complex on the vertex set $V = \{x_1, . . . , x_n \}$. Every element of $\Delta$ is called a face of $\Delta$ and a facet of $\Delta$ is a maximal face of $\Delta$ with respect to inclusion. If all facets of $\Delta$ have the same cardinality, then $\Delta$ is called pure. For  the simplicial complex $\Delta$, we may consider a square-free monomial ideal  $I = I_{\Delta}$ of $R$ which is generated by all minimal nonface of $\Delta$ is called the Stanley-Reisner ideal of $\Delta$ and $K[\Delta] = R/I_{\Delta}$ is called the Stanley-Reisner ring.
For the simplicial complex $\Delta$ and $F \in \Delta$, link of $F$ in $\Delta$ is defined as $\lk_{\Delta}(F) = \{G \in \Delta ~| ~G \cap  F = \emptyset, G \cup  F \in \Delta \}$. If $\Delta$ is a simplicial complex with facets $F_1,\ldots, F_t$, we denote $\Delta$ by $ \langle F_1, \ldots, F_t \rangle $,
and $\{F_1, \ldots, F_t \}$ is called the facet set of $\Delta$.

 \begin{Proposition}\label{MV}
 Let $\Delta_1$ and $\Delta_2$ be two simplicial complexes on $[n]$, and let $\Delta = \Delta_1 \cup \Delta_2$ and $\Gamma= \Delta_1 \cap \Delta_2$. Then there exists an exact sequence of the following form
 \begin{align*}
\ldots  &\longrightarrow \tilde{H}_{l}(\Gamma; k) \longrightarrow \tilde{H}_{l}(\Delta_1; k)\oplus \tilde{H}_{l}(\Delta_2; k) \longrightarrow \tilde{H}_{l}(\Delta; k) \\
&\longrightarrow \tilde{H}_{l-1}(\Gamma; k) \longrightarrow \tilde{H}_{l-1}(\Delta_1; k)\oplus \tilde{H}_{l-1}(\Delta_2; k) \longrightarrow \tilde{H}_{l-1}(\Delta; k) \\
&\longrightarrow \ldots
\end{align*}
This sequence is called the reduced Mayer-Vietoris exact sequence (\cite[Proposition 5.1.8]{HH1}).
\end{Proposition}

\begin{Lemma}\label{L0}(\cite[Lemma 2.5]{VV})
Let $x$ be a vertex of $G$ and $G^{'}=G \setminus N_G[x]$. Then $\Delta_{G^{'}}=lk_{\Delta_G}\{x\}$. In particular, $F$ is a facet of $\Delta_{G^{'}}$ if and only if $x \notin F$ and $F \cup \{x \}$ is a facet of $\Delta_G$.

\end{Lemma}

The family of all unmixed bipartite graphs has
been characterized in a combinatorial way by Villarreal in the following result.
\begin{Theorem}\label{T1}(\cite[Theorem 1.1]{V1})
Let $G$ be a bipartite graph without an isolated
vertex. Then $G$ is unmixed if and only if there is a partition $V_1 = \{ x_1, \ldots , x_n \}$ and
$V_2 = \{ y_1, \ldots , y_n \}$ of vertices of $G$ such that
\begin{enumerate}
\item $\{ x_i ,y_i  \} \in E(G) $ for $1 \leq i \leq n$ and

\item if $\{ x_i, y_j \}$ and $\{ x_j, y_k \}$ are edges in $G$, for some distinct $i, j$ and $k$, then $\{ x_i,y_k \} \in E(G)$.
\end{enumerate}
\end{Theorem}
In this case, such a partition and ordering is called a pure order of $G$. As stated before, if $G$ is an unmixed bipartite graph on the vertex set $V (G) = \{ x_1, \ldots , x_n \} \cup \{y_1, \ldots , y_n \}$, then each of its minimal vertex covers has the form
$\{x_{i_1} , \ldots , x_{i_s} , y_{i_{s+1}}, \ldots , y_{i_n} \}$, where $\{ i_1, \ldots , i_n \} = [n]$.

\section{Almost Cohen-Macaulay Bipartite graphs}

We start this section by the following lemma.

\begin{Lemma}\label{L3}
Let  Let G be a bipartite graph with bipartition $V_1=\{x_1,\ldots, x_n\}$ and $V_2=\{y_1,\ldots,y_m\}$  and let $H=G \setminus N_{G}[x_{i_1}, \ldots,x_{i_s}]$ be subgraph of $G$. Then
\[ \Delta_{H}=\lk_{\Delta_{G}} \{ x_{i_1}, \ldots, x_{i_s}\} \]

\end{Lemma}
\begin{proof}
We use induction on $s$. When $s = 1$, the result follows from Lemma \ref{L0}. Now suppose, inductively, that $s >1$ and the result has been proved for smaller values of $s$.
	Let $K=G\setminus N_{G}[x_{i_1},\ldots, x_{i_{s-1}} ]$, by induction hypothesis $\Delta_{K}=\lk_{\Delta_{G}} \{ x_{i_1}, \ldots, x_{i_{s-1}} \}$. Set $\Gamma=\lk_{\Delta_{G}} \{x_{i_1}, \ldots, x_{i_{s-1}} \} $. Since $ H= K \setminus N_{K}[ x_{i_s} ]$ and $\lk_{\Gamma}\{x_{i_{s}}\}=\lk_{\Delta_{G}}\{x_{i_1}, \ldots, x_{i_s} \}$, we have $ \Delta_{H}=\lk_{\Delta_{G}} \{ x_{i_1}, \ldots, x_{i_s} \}$.
\end{proof}

Let $I$ be a monomial ideal of $R$. We denote, as usual, by $G(I)$ the unique minimal set of monomial generators of $I$. If $I$ is generated in a single degree, then $I$ is said to be {\it polymatroidal} if for any two elements $u, v \in G(I)$ such that $\deg_{x_{i}}(v) < \deg_{ x_{i}}(u)$ there exists an index $j$ with $\deg_{x_{j}} (u) < \deg_{x_{j}} (v)$ such that $x_{j}(u/x_{i}) \in G(I)$. The polymatroidal ideal $I$ is called {\it matroidal} if $I$ is generated by square-free monomials (see \cite{HH1}).
\begin{Lemma}\label{L1}
Let $K_{m,n}$ be complete bipartite graph. Then $K_{m,n}$ is aCM if and only if $n \leq m \leq 2$.
\end{Lemma}
\begin{proof}
$(\Longleftarrow).$ This is obvious.

$(\Longrightarrow). $
Let $K_{m,n}$ be aCM complete bipartite graph and $n \leq m $. So the edge ideal of $K_{m,n}$ is of the form $I(G)=(x_1, \ldots, x_m )(y_1, \ldots, y_n)$. $I(G)$ is transversal matroidal ideal of degree $2$. Therefore by \cite{C} we have $\depth(R/I)=1$. If $\dim(R/I)=1$, then $I(G)$ is CM and by \cite[Remark 2.2]{EV} we have $n=m=1$ . If $\dim(R/I)=2$, then $\hte(I(G))=n+m-2$. Since $\hte(I(G)) \leq n$, we have $n+m-2 \leq n$. Hence we must have $n \leq m \leq 2$.
\end{proof}

Estrada and Villarreal in \cite{EV} proved that if $G$ is a CM bipartite graph, then there is a vertex $v\in V(G)$ such that $deg(v)=1$ and also they proved that $G\setminus\{u\}$ is CM for some vertex $u\in V(G)$. Now we generalize that results.
\begin{Theorem}\label{T3}
Let $G$ be an aCM bipartite graph with bipartition $V_1= \{ x_1, \ldots, x_n \}$ and $V_2= \{ y_1, \ldots, y_m \}$. If $G$ has at least one vertex of positive degree, then there is a vertex $v \in V_1 \cup V_2$ so that $\deg(v) \leq 2$.
\end{Theorem}
\begin{proof}
Since $G$ is aCM bipartite graph, by \cite[Corollary 2.4]{MN2} we may assume that $m-1 \leq n \leq m$. First assume $n=m-1$.
We proceed by contradiction. Assume $\deg(v) \geq 3$ for all $v \in V_1 \cup V_2$. Let $v$ be a vertex of minimal degree. We may assume $v = x_n$. If $\deg(x_n) = n+1$, then $G=K_{n,n+1}$ is a complete bipartite graph, but this is impossible by Lemma \ref{L1}.
Therefore we may assume that $3 \leq \deg(x_n) \leq n$. Set $s = \deg(x_n)$. Let $N_{G}(x_n)=\{ y_{n+1}, \ldots, y_{n-s+2} \}$. Let $\Delta $ be the simplicial complex of independent sets of $G$, by Lemma \ref{L3} we have $\Gamma= \lk_{\Delta} \{x_n \}$ is the simplicial complex of independent sets of $G\setminus N_{G}[x_{n} ] $. Since $\Delta$ is aCM,  by \cite[Theorem 3.4]{MN2} $\Gamma$ is C-M for $\{x_n\} \in \Delta \setminus \Delta(n)$ and $\Gamma$ is aCM for $\{x_n\} \in \Delta(n)$.

 If $ \{x_n\} \in \Delta(n)$, then $\dim \Gamma =n-1$. Therefore we have $n-1 \leq |F| \leq n$ for all facet $F$ in $\Gamma$. We claim $x_{n-s+2}, \ldots, x_{n-1}$ are isolated vertex in $G \setminus  N_{G}[x_n]$. By contrary we assume one of them is not isolated vertex in $G \setminus  N_{G}[x_n]$, say $x_{n-1}$. So $\{ x_{n-1}, y_j \} \in E(G)$ for some $1\leq j \leq n-s+1$. Therefore $\{ y_1, \ldots, y_{n-s+1}, x_{n-s+2}, \ldots, x_{n-2} \}$ is maximal facet of bipartite graph $G\setminus N_{G}[x_{n} ] $, and this is contradiction, since $|\{ y_1, \ldots, y_{n-s+1}, x_{n-s+2}, \ldots, x_{n-2} \} |=n-2$. So $\deg(x_{n-s+2})= \ldots= \deg(x_{n-1})=s$ and $N_{G}(x_{n-s+2})= \ldots=N_{G}(x_{n-1}) =\{ y_{n+1}, \ldots, y_{n-s+2} \}$. Consider the graph $H = G \setminus  N_{G}[ y_1, \ldots, y_{n-s+1} ]$. $H$ is aCM bipartite graph, since by Lemma \ref{L3} $\lk_{\Delta} \{ y_1, \ldots, y_{n-s+1} \}$ is the simplicial complexs of independent sets of $H$ but this is impossible because $H=K_{s-1,s}$.

 If $\{x_{n}\} \notin \Delta(n) $, then $\dim \Gamma =n-2$. Since $\Gamma$ is CM, we have $|F| = n-1$ for all facet $F$ in $\Gamma$. We claim $x_{n-s+2}, \ldots, x_{n-1}$ are isolated vertex in $G \setminus  N_{G}[x_n]$.  By contrary assume one of them is not isolated vertex in $G \setminus  N_{G}[x_n]$, say $x_{n-1}$. So $\{ x_{n-1}, y_j \} \in E(G)$ for some $1\leq j \leq n-s+1$. Therefore $\{ y_1, \ldots, y_{n-s+1}, x_{n-s+2}, \ldots, x_{n-2} \}$ and $\{ x_1, \ldots, x_{n-1}\}$ are maximal facet of bipartite graph $G\setminus N_{G}[x_{n} ]$, and this is contradiction. So $\deg(x_{n-s+2})= \ldots= \deg(x_{n-1})=s$ and $N_{G}(x_{n-s+2})= \ldots=N_{G}(x_{n-1}) =\{ y_{n+1}, \ldots, y_{n-s+2} \}$. Consider the graph $H = G \setminus  N_{G}[ y_1, \ldots, y_{n-s+1} ]$. $H$ is aCM bipartite graph, since by Lemma \ref{L3} $\lk_{\Delta} \{ y_1, \ldots, y_{n-s+1} \}$ is the simplicial complexs of independent sets of $H$ but this is impossible because $H=K_{s-1,s}$. If $|V_1|=|V_2|=m=n$, then the proof is exactely the same of the above  arguments.
\end{proof}

\begin{Corollary}
Let $G$ be an aCM bipartite graph. Let $u$ be a vertex of degree one
of $G$ and $v$ its adjacent vertex. Then $G \setminus \{ v\}$ is aCM.

\end{Corollary}
\begin{proof}
Let $H=G\setminus \{ u, v \}$ be subgraph of $G$, by Lemma \ref{L3} we have $\lk_{\Delta} \{u \}$ is the
simplicial complexs of independent sets of $H$. By \cite[Corollary 3.6]{MN2}, $\lk_{\Delta} \{u \}$ is aCM. If $\Gamma = \langle F_1, \ldots, F_r \rangle$ be simplicial complex of independent set of $G\setminus\{ v\}$, then $ F_i=G_i \cup \{u \}$ such that $G_i$ is a facet of $\lk_{\Delta} \{u \}$. Hence $G\setminus\{ v\}$ is aCM.
\end{proof}

\begin{Corollary}
Let $G$ be an aCM bipartite graph. Let $u$ be a vertex of degree two
of $G$ and $v, w$ its adjacent vertex. then $G \setminus \{ v, w\}$ is aCM.

\end{Corollary}
\begin{proof}
Let $H=G\setminus \{ u, v, w \}$ be subgraph of $G$, by Lemma \ref{L3} we have $\lk_{\Delta} \{u \}$ is the
simplicial complexs of independent sets of $H$. By \cite[Corollary 3.6]{MN2}, $\lk_{\Delta} \{u \}$ is aCM. If $\Gamma = \langle F_1, \ldots, F_r \rangle$ be simplicial complex of independent set of $G\setminus \{v, w \}$, then $ F_i=G_i \cup \{u \}$ such that $G_i$ is a facet of $\lk_{\Delta} \{u \}$. Hence $G \setminus \{ v, w\}$ is aCM.
\end{proof}

\begin{Lemma}\label{L2}
Let $G$ be an unmixed bipartite graph with bipartition $V_1= \{ x_1, \ldots, x_n \}$ and $V_2= \{ y_1, \ldots, y_n \}$ and let $N_{G}(x_{n}) =\{ y_{n}, \ldots, y_{n-i+1} \}$. Then $\lk_{\Delta_G} \{y_{n}, \ldots, y_{n-i+1} \}$ is subsimplicial complex of $\lk_{\Delta_G} \{ x_n \}$.
\end{Lemma}
\begin{proof}
Let $H=G\setminus N_{G}[x_{n} ]$ and $K=G\setminus N_{G}[y_{n}, \ldots, y_{n-i+1}]$ be subgraphs of $G$, by Lemma \ref{L3}, we have $\lk_{\Delta_G} \{x_n \}$ and $\lk_{\Delta_G} \{y_{n}, \ldots, y_{n-i+1} \}$ are the
simplicial complexs of independent sets of $H$ and $K$ respectively.
If $F \in \lk_{\Delta_G}\{y_{n}, \ldots, y_{n-i+1} \}$, then $y_j \notin F$ for all $n-i+1 \leq  j \leq n$ and $F \cup \{y_{n}, \ldots, y_{n-i+1}\} \in {\Delta_G}$. This implies that $F \cup \{ y_{n}, \ldots, y_{n-i+1} \}$ is an independent set of $G$. So $(F \cup \{ y_{n}, \ldots, y_{n-i+1} \}) \cap N_{G}(\{y_{n}, \ldots, y_{n-i+1} \}) = \emptyset$. But this means that $ F \subseteq V(K) \subseteq V(H)$ because $V(K)=V(G) \setminus N_{G}[y_{n}, \ldots, y_{n-i+1}]$. Since $\{y_{n}, \ldots, y_{n-i+1} \} =N_{G}(x_n)$ and $x_n \in N_{G}(y_{n}, \ldots, y_{n-i+1})$, we have $(F \cup N_{G}(x_n)) \cap \{x_n \}= \emptyset$. Thus $(F \cup \{x_n \}) \cap N_{G}(x_n)= \emptyset$. Hence $F \in \lk_{\Delta_G}\{x_n \}$.
\end{proof}

\begin{Lemma}\label{L6}
Let $G$ be an unmixed bipartite graph with pure order of vertices $\{x_1, \ldots , x_n \} \cup \{y_1, \ldots , y_n \}$ and let $N_G(x_i)=\{ y_i,y_{i_1}, \ldots, y_{i_{r_i}} \}$. Then
\begin{enumerate}
\item[(i)] $G\setminus N_{G}[x_i]$ is unmixed bipartite subgraph of $G$.
\item[(ii)] $x_{i_1}, \ldots, x_{i_{r_i}}$ are isolated vertices in $G\setminus N_{G}[x_i]$.
\end{enumerate}
In particular if $x_i$ is a vertex of minimal degree, then
\begin{enumerate}
\item[(iii)]  $G\setminus N_{G}[ \{ y_1, \ldots, y_n \} \setminus \{ y_i, y_{i_1}, \ldots, y_{i_{r_i}} \}]$ is complete bipartite graph with bipartition
$\{x_i, x_{i_1}, \ldots, x_{i_{r_i}} \} \cup \{ y_i, y_{i_1}, \ldots, y_{i_{r_i}} \}$
\item[(iv)] $N_G(y_i)=N_G(y_{i_1})= \ldots = N_G(y_{i_{r_i}})$
\end{enumerate}

\end{Lemma}

\begin{proof}
$(i)$ Since $\lk_{\Delta_G} \{x_i \}$ is the simplicial complexs of independent sets of $G\setminus N_{G}[x_i]$ and any link of a pure complex is pure, we have $G\setminus N_{G} [x_i]$ is unmixed bipartite subgraph of $G$.\\
$(ii)$ If $x_j$ for some $ i_1 \leq j \leq i_{r_i}$ is not isolated in $G\setminus N_{G}[x_i]$,
then there exists an integer $k \in \{ 1, \ldots, n \} \setminus \{ i_1, \ldots, i_{r_i} \}$ such that $x_j y_k \in E(G\setminus N_{G}[x_i])$.
Therefore $\{ y_1, \ldots, y_n \} \setminus \{ y_i, y_{i_1}, \ldots, y_{i_{r_i}} \}$ is
a minimal vertex cover for $G\setminus N_{G}[x_i]$ and there is also a minimal vertex cover for $G\setminus N_{G}[x_i]$ containing $(\{ x_1, \ldots, x_n \} \setminus \{ x_i, x_{i_1}, \ldots, x_{i_{r_i}} \})\cup \{x_j \} $, which is a contradiction since $G\setminus N_{G}[x_i]$ is unmixed. \\
$(iii)$ Since $x_{i_1}, \ldots, x_{i_{r_i}} $ are isolated vertices in $G\setminus N_{G}[x_i]$, we have $\deg x_{i_j} \leq r_{i} +1$ and $N_G(x_{i_j}) \subseteq N_G(x_i)$ for all $i_1 \leq j \leq i_{r_{i}}$. By hypothesis $x_i$ is a vertex of minimal degree, therefore  $\deg x_{i_j} \geq r_{i} +1$ for all $i_1 \leq j \leq i_{r_{i}}$. Hence $G\setminus N_{G}[ \{ y_1, \ldots, y_n \} \setminus \{ y_i, y_{i_1}, \ldots, y_{i_{r_i}} \}]$ is complete bipartite graph with bipartition
 $\{x_i, x_{i_1}, \ldots, x_{i_{r_i}} \} \cup \{ y_i, y_{i_1}, \ldots, y_{i_{r_i}} \}$.\\
 $(iv)$ Let  $x_k \in N_G(y_{i_j})$ for some $ j \in \{ i, i_1, \ldots, i_{r_i} \}$, say $x_k \in N_G(y_{i_1})$. Therefore $ \{x_k , y_{i_1} \} \in E(G)$. Since $N_G(x_{i_1})=\{ y_i, y_{i_1}, \ldots, y_{i_{r_i}} \}$, by pure order of vertices we have $ \{x_k , y_{i_j} \} \in E(G)$ for all $ j \in \{ i, i_1, \ldots, i_{r_i} \}$. Hence $N_G(y_i)=N_G(y_{i_1})= \ldots = N_G(y_{i_{r_i}})$ .

\end{proof}

\begin{Lemma}\label{L7}
Let $G$ be an unmixed bipartite graph with pure order of vertices $\{x_1, \ldots , x_n \} \cup \{y_1, \ldots , y_n \}$ and let $H=G\setminus N_{G}[x_{i} ]$ and $K=H \setminus N_G[\{x_{i_1},\ldots, x_{i_{r_i}}\}]$ be subgraphs of $G$ such that $x_i$ is a vertex of minimal degree and $N_G(x_i)=\{ y_i,y_{i_1}, \ldots, y_{i_{r_i}} \}$. Then $F$ is a facet
of $\Delta_{G}$ if and only if  $F$ can be written as
\begin{enumerate}
\item $F= F^{\prime} \cup \{x_{i} \}$, where $F^{\prime}$ is a facet of $\Delta_H$.
\item $F= F^{\prime\prime} \cup \{y_{i},y_{i_1}, \ldots, y_{i_{r_i}} \}$, where $F^{\prime \prime}$ is a facet of $\Delta_K$.
\end{enumerate}
\end{Lemma}
\begin{proof}
Let $F$ be a facet of $\Delta_G$. Since $G$ is unmixed bipartite graph, it follows that if $x_i \notin F$, then $y_i \in F$. By Lemma \ref{L6} part $(iv)$, we have $N_G(y_i)=N_G(y_{i_1})= \ldots = N_G(y_{i_{r_i}})$ and so $y_{i_1}, \ldots, y_{i_{r_i}} \in F$. Set $F^{\prime\prime}=F \setminus \{ y_{i},y_{i_1}, \ldots, y_{i_{r_i}} \}$. We must show that $F^{\prime\prime}$ is a facet of $\Delta_K$. First notice that $F^{\prime\prime}$ is an
independent set of $K$, because $E(K) \subseteq E(G)$. Let $L$ be a facet of $\Delta_K$ containing $F^{\prime\prime}$. Since
$L \cup \{y_{i},y_{i_1}, \ldots, y_{i_{r_i}} \}$ is an independent set of $G$, and $G$ is unmixed, we obtain that $|L|+r_{i}+1 \leq |F| =|F^{\prime\prime}|+r_{i}+1$. Hence $F^{\prime\prime}=L$.

On the other hand if $x_i \in F$,
then we have $y_{i},y_{i_1}, \ldots, y_{i_{r_i}} \notin F$. Set $F^{\prime}=F \setminus \{ x_{i} \}$. $F^{\prime}$ is a facet of $\Delta_{H}$, since $\Delta_{H}=\lk_{\Delta_G}{ \{ x_i \} }$.
 The converse also follows readily by using the similar arguments.
\end{proof}

\begin{Theorem}\label{T4}
Let $G$ be an unmixed bipartite graph with bipartition $V_1= \{ x_1, \ldots, x_n \}$ and $V_2= \{ y_1, \ldots, y_n \}$ and let $H=G\setminus N_{G}[x_{n} ]$ and $K=G\setminus N_{G}[y_{n}]$ be subgraphs of $G$ such that $N_{G}(x_{n}) =\{ y_{n} \}$. Then $H$ and $K$ are aCM graphs if and only if $G$ is aCM graph.
\end{Theorem}
\begin{proof}
$(\Longleftarrow).$
Let $\Delta =\langle F_1, F_2, \ldots, F_m \rangle$ be the simplicial complex of independent sets of $G$. By Lemma \ref{L3} we have $\lk_{\Delta} \{x_n \}$ and $\lk_{\Delta} \{y_{n} \}$ are the
simplicial complexs of independent sets of $H$ and $K$ respectively. Therefore by\cite[Corollary 3.6]{MN2}, $H$ and $K$ are aCM.

$(\Longrightarrow).$
By \cite[Theorem 3.4]{MN2}, it is enough to show that $\tilde{H}_{i}(\lk_{\Delta}F; k)=0$ for all $F \in \Delta(n-1)$ and for all $i < \dim \lk_{\Delta}F-1$. We proceed by induction on $\dim \lk_{\Delta}F$. The case $\dim \lk_{\Delta}F=1$ is clear. Assume that $\dim \lk_{\Delta}F > 1$ and the assertion holds for $\dim \lk_{\Delta}F=n-1$.
Since for each $l$, $\{ x_l, y_l \} \in E(G)$ and $G$ is unmixed, every facet $F_k$ has exactely one of
$x_l$ and $y_l$. Thus either $x_l \in F_k$ or $y_l \in F_k$.
Since $N_{G}(x_{n}) =\{ y_{n} \}$, we can assume $ F_i$ contains $x_n$ such that $y_n \notin F_i$ for $i=1, \ldots , s$ and $ F_i$ contains $y_n$ such that $x_n \notin F_i$ for $i=s+1, \ldots, m$. Let $\Delta_{1}=\langle F_1, \ldots, F_s \rangle$ and $\Delta_{2}=\langle  F_{s+1}, \ldots, F_m \rangle$. Then $\Delta= \Delta_1 \cup \Delta_2$.
 We claim $\Delta_1 \cap \Delta_2=\lk_{\Delta} \{y_n \}$. It is obvious that $\lk_{\Delta} \{y_n \}$ is a subsimplicial complex of $\Delta_1 \cap \Delta_2$, since by Lemma \ref{L2}, $\lk_{\Delta} \{y_n \}$ is subsimplicial complex of $\lk_{\Delta} \{ x_n \}$. Now assume $F \in \Delta_1 \cap \Delta_2$. If $ F \in \Delta_1$, then $F \subseteq F_i$ for some $i$ in $\{ 1, \ldots, s \}$. By Lemma \ref{L7}, there exist a facet $F^{\prime}_{i}$ in $\lk_{\Delta} \{x_n \}$ such that $F \subseteq F^{\prime}_{i} \cup \{ x_n \}$. By similar argument there exist a facet $F^{\prime}_{j}$ in $\lk_{\Delta} \{y_{n} \}$ such that $F \subseteq F^{\prime}_{j} \cup \{ y_n \}$. Hence $ F \subseteq F^{\prime}_{i}\cap F^{\prime}_{j} \in \lk_{\Delta} \{y_n \}$. Now by using the Mayer-Vietoris exact sequence we have $\tilde{H}_{i}({\Delta}; k)=0$ for $i < n-2$. Hence $\Delta$ is aCM.
\end{proof}

\begin{Theorem}\label{T5}
Let $G$ be an unmixed bipartite graph with pure order of vertices $ \{ x_1, \ldots, x_n \} \cup \{ y_1, \ldots, y_n \}$. Let $H=G\setminus N_{G}[x_{n} ]$ and $K=G\setminus N_{G}[y_{n},  y_{n-1} ]$ be subgraphs of $G$ such that $N_{G}(x_{n}) =\{ y_{n}, y_{n-1} \}$. If  $H$ is aCM and $K$ is CM, then $G$ is aCM.
\end{Theorem}
\begin{proof}

Let $\Delta =\langle F_1, F_2, \ldots, F_m \rangle$ be the simplicial complex of independent sets of $G$. By Lemma \ref{L3},  we have $\lk_{\Delta} \{x_n \}$ and $\lk_{\Delta} \{y_{n}, y_{n-1} \}$ are simplicial complexs of independent sets of $H$ and $K$, respectively.

By \cite[Theorem 3.4]{MN2}, it is enough to show that $\tilde{H}_{i}(\lk_{\Delta}F; k)=0$ for all $F \in \Delta(n-1)$ and for all $i < \dim \lk_{\Delta}F-1$. We proceed by induction on $\dim \lk_{\Delta}F$. The case $\dim \lk_{\Delta}F=1$ is clear. Assume that $\dim \lk_{\Delta}F > 1$ and the assertion holds for $\dim \lk_{\Delta}F=n-1$.
Since $N_{G}(x_{n}) =\{ y_{n}, y_{n-1} \}$ and $G$ is unmixed, every facet $F_k$  has exactely one of
$x_n$ and $y_i$ for $n-1 \leq i \leq n$. Now by Lemma \ref{L6}, we can assume $ F_i$ contains $x_n$ and none of $ y_j \in \{y_n, y_{n-1} \}$ for $i=1, \ldots , s$ and also $\{y_{n}, y_{n-1} \} \subseteq F_j$ for $j=s+1, \ldots, m$ such that $x_n \notin F_j$. Let $\Delta_{1}=\langle F_1, \ldots, F_s \rangle$ and $\Delta_{2}=\langle  F_{s+1}, \ldots, F_m \rangle$. Then $\Delta= \Delta_1 \cup \Delta_2$.
By Lemma \ref{L3}, $\lk_{\Delta} \{y_{n}, y_{n-1} \}$ is subsimplicial complex of $\lk_{\Delta} \{x_n \}$. We claim $\Delta_1 \cap \Delta_2=\lk_{\Delta} \{y_{n}, y_{n-1} \}$. It is obvious that $\lk_{\Delta} \{y_{n}, y_{n-1} \}$ is a subsimplicial complex of $\Delta_1 \cap \Delta_2$. Now assume $F \in \Delta_1 \cap \Delta_2$. If $ F \in \Delta_1$, then $F \subseteq F_i$ for some $i$ in $\{ 1, \ldots, s \}$. By Lemma \ref{L7}, there exist a facet $F^{\prime}_{i}$ in $\lk_{\Delta} \{x_n \}$ such that $F \subseteq F^{\prime}_{i} \cup \{ x_n \}$. By similar argument there exist a facet $F^{\prime}_{j}$ in $\lk_{\Delta} \{y_{n}, y_{n-1} \}$ such that $F \subseteq F^{\prime}_{j} \cup \{ y_n,y_{n-1} \}$. Hence $ F \subseteq F^{\prime}_{i}\cap F^{\prime}_{j} \in \lk_{\Delta} \{y_{n}, y_{n-1} \}$. By using the Mayer-Vietoris exact sequence, it therefore follows that $\Delta$ is aCM.
\end{proof}

The following example shows that the condition of Cohen-Macualyness of $K$ is essential.

\begin{Example}
Let $G$ be the following unmixed bipartite graph.
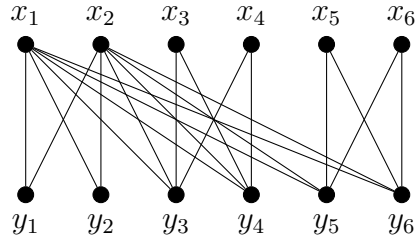
\[\begin{tikzpicture}
	\vertex[fill] (x1) at (0,2) [label=above:$x_{1}$] {};
	\vertex[fill] (x2) at (1,2) [label=above:$x_{2}$] {};
	\vertex[fill] (x3) at (2,2) [label=above:$x_{3}$] {};
\vertex[fill] (x4) at (3,2) [label=above:$x_{4}$] {};
\vertex[fill] (x5) at (4,2) [label=above:$x_{5}$] {};
\vertex[fill] (x6) at (5,2) [label=above:$x_{6}$] {};
	\vertex[fill] (y1) at (0,0) [label=below:$y_{1}$] {};
	\vertex[fill] (y2) at (1,0) [label=below:$y_{2}$] {};
	\vertex[fill] (y3) at (2,0) [label=below:$y_{3}$] {};
\vertex[fill] (y4) at (3,0) [label=below:$y_{4}$] {};
\vertex[fill] (y5) at (4,0) [label=below:$y_{5}$] {};
\vertex[fill] (y6) at (5,0) [label=below:$y_{6}$] {};
	\path
		(x1) edge (y1)
		(x1) edge (y2)
		(x1) edge (y3)
		(x1) edge (y4)
		(x1) edge (y5)
		(x1) edge (y6)
		(x2) edge (y1)
		(x2) edge (y2)
		(x2) edge (y3)
(x2) edge (y4)
(x2) edge (y5)
(x2) edge (y6)
		(x3) edge (y3)
		(x3) edge (y4)
		(x4) edge (y3)
		(x4) edge (y4)
		(x5) edge (y5)
		(x5) edge (y6)
		(x6) edge (y5)
		(x6) edge (y6)
		
	;
\end{tikzpicture}\]
$H=G \setminus N_G [x_6]$ and $K=G \setminus N_{G}[ \{ y_5,y_6 \}]$ are aCM, but $G$ is not aCM.
\end{Example}

\section{Almost Cohen-Macaulay and connected in codimension two}

We recall the concept of connected in codimension $k$, for a topological space is defined by Hartshoren in \cite{Har}, for any non-negative integer $k$. For a monomial ideal $I$, considering the Zariski topology on $\Spec(R/I)$, we get that the closed subsets in this topology are the sets
$V(J) = \{ \frk{q} \in \Spec (R) ~|~  J \subseteq \frk{q} \}$, where $J \subseteq I$ is an ideal of $R$ and the irreducible
components of $\Spec (R/I)$ are the closed sets $V(\frk{p})$, where $\frk{p}$ is a minimal prime ideal of
$I$. $\Spec (R/I)$ with this topology is a connected space. By \cite[Proposition 1.1]{Har}, the ideal $I$ is called connected in codimension $k$, if $V(\frk{p})$ and $V(\frk{q})$ are irreducible comnponents of $\Spec (R/I)$, then there is a finite sequence $V(\frk{p}) =V(\frk{p}_1),V(\frk{p}_2) , \ldots, V(\frk{p}_r)= V(\frk{q})$ of irreducible components of $\Spec (R/I)$, such that for each $i= 1, 2, \ldots , r -1$, $V(\frk{p}_i) \cap V(\frk{p}_{i+1})$ is of codimension $\leq k$ in $\Spec (R/I)$. Since the codimension of $V(\frk{p})$ is equal to $\hte (\frk{p})- \hte (I)$
for all prime ideals $\frk{p} \supseteq I$, a monomial ideal $I \subset R$ is connected in codimension $k$, if for any pair of distinct minimal prime ideals $\frk{p}$ and $\frk{q}$, there exists a  sequence of minimal prime ideals $\frk{p}=\frk{p}_1, \ldots, \frk{p}_r=\frk{q}$ such that $\hte(\frk{p}_i +\frk{p}_{i+1})\leq \hte(I)+k$. As $\frk{p}_i +\frk{p}_{i+1}$ is again a prime ideal, one can replace the height with the number of generators.

The definition for a simplicial complex $\Delta$ should be organized in such a way that the Stanley-Reisner ideal of $\Delta$, becomes connected in codimension $k$, i.e. A simplicial complex $\Delta$ is said to be connected in codimension $k$, if for any two facets $F$ and $G$ of $\Delta$, there is
a sequence of facets $F = F_1, F_2, . . . , F_r = G$ such that $\dim (F_i \cap F_{i+1}) \geq \dim \Delta - k$, for each $1 \leq i \leq r - 1$.
By the above definition,  if a monomial ideal $I$ is connected in codimension one, then it is equidimensional i.e. all minimal prime ideals of $I$ have the same height. In particular, if $I$ is square-free monomial ideal connected in codimension one then it is unmixed i.e. all prime ideals of $\Ass(I)$ have the same height,
(see also \cite[Definition 3.1]{BJ}).

Thus we can rewrite the following result:

\begin{Theorem}(\cite[Theorem 1.3]{HYZ})
Let $G$ be a bipartite graph with at least four vertices. Then $G$ is a connected in codimension one if and only if $G$ is a CM graph.
\end{Theorem}

\begin{Theorem}\label{T2}
Let $\Delta$ be $(d-1)$-dimensional ($d \geq 3$) aCM simplicial complex. Then it is connected in codimension two.
\end{Theorem}
\begin{proof}
We argue by induction on $d$. If  $d=3$, then by \cite[Corollary 3.5]{MN2} $\Delta$ is connected. Therefore for arbitrary facets $F$ and $E$, there exists a
sequence of facets $F = F_0, F_1, . . . , F_{n-1}, F_n = E$ such that $F_i \cap F_{i+1}  \ne \emptyset$. So $\dim(F_{i} \cap F_{i+1}) \geq 0= \dim \Delta -2$. Hence $\Delta$ is connected in codimension two.
Now we assume that $d > 3$. Let $F$
and $E$ be two facets of $\Delta$. Since $\Delta$ is aCM and $d \geq 3$, we have $\Delta$ is connected by \cite[Lemma 3.2 and Theorem 3.4]{MN2}. Therefore  there exists a
sequence of facets $F = F_0, F_1, . . . , F_{n-1}, F_n = E$ such that $F_i \cap F_{i+1}  \ne \emptyset$. Let
$x_i$ be a vertex belonging to $F_i \cap F_{i+1}$. Since $\Delta$ is aCM, $\lk_{\Delta}\{ {x_i} \}$
is CM for $\{x_i\} \in \Delta \setminus \Delta(d-1)$ and $\lk_{\Delta}\{ {x_i} \}$ is aCM for $\{x_i\} \in \Delta(d-1)$ by \cite[Corollary 3.6]{MN2}. By \cite[Lemma 9.1.12]{HH1} and working with induction on the dimension of $\Delta$, we may assume that $\lk_{\Delta}\{ {x_i} \}$ is connected in codimension one for $\{x_i\} \in \Delta \setminus \Delta(d-1)$ and $\lk_{\Delta}\{ {x_i} \}$ is connected in codimension two  for $\{x_i\} \in \Delta(d-1)$ respectively. Thus $F_{i}^{\prime}:=F_{i} \setminus \{x_i \}$ and $F_{i+1}^{\prime}:=F_{i+1} \setminus \{x_i \}$ are facets of $\lk_{\Delta}\{ {x_i} \}$ and therefore there exists a sequence of facets $F_{i}^{\prime}= H_{0}^{\prime}, H_{1}^{\prime}, \ldots, H_{r-1}^{\prime}, H_{r}^{\prime} = F_{i+1}^{\prime}$
of $\lk_{\Delta}\{ {x_i} \}$ such that $|H_{j}^{\prime} \cap H_{j+1}^{\prime}| \geq d  -3$. Set $H_{j}=H_{j}^{\prime} \cup \{x_i \}$.
So there exists a sequence of facets $F_i = H_0, H_1, \ldots, H_{r-1}, H_r = F_{i+1}$
of $\Delta$, where all $H_j$ contain $x_i$ with $|H_j| \geq d-1$, such that $|H_j \cap H_{j+1}| \geq d  -2$.
Composing all these sequences of facets which we have between each $F_i$ and
$F_{i+1}$ yields the desired sequence between $F$ and $E$. This completes the proof.
\end{proof}

As before if the monomial ideal $I$ is connected in codimension one, then $I$ is equidimensional. But the following example says that for connected in codimensional two this is false.
\begin{Example}
Let $G$ be the following graph. Then
\[I(G)=(x_{1}y_{1}, x_{1}y_{2}, x_{2}y_1,x_2y_2, x_2y_3, x_3y_1, x_3y_2,x_3y_3,x_3y_4,x_4y_1,x_4y_2,x_4y_3,x_4y_4).\]

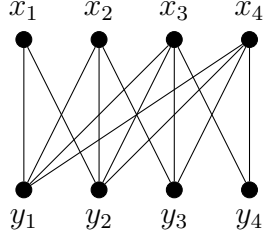
\[\begin{tikzpicture}
	\vertex[fill] (x1) at (0,2) [label=above:$x_{1}$] {};
	\vertex[fill] (x2) at (1,2) [label=above:$x_{2}$] {};
	\vertex[fill] (x3) at (2,2) [label=above:$x_{3}$] {};
\vertex[fill] (x4) at (3,2) [label=above:$x_{4}$] {};
\vertex[fill] (y1) at (0,0) [label=below:$y_{1}$] {};
	\vertex[fill] (y2) at (1,0) [label=below:$y_{2}$] {};
	\vertex[fill] (y3) at (2,0) [label=below:$y_{3}$] {};
\vertex[fill] (y4) at (3,0) [label=below:$y_{4}$] {};
\path
		(x1) edge (y1)
		(x1) edge (y2)

		(x2) edge (y1)
		(x2) edge (y2)
		(x2) edge (y3)
(x3) edge (y1)
		(x3) edge (y2)
		(x3) edge (y3)
		(x3) edge (y4)
		(x4) edge (y1)
		(x4) edge (y2)
		(x4) edge (y3)
		(x4) edge (y4)

	;
\end{tikzpicture}\]
By Macaulay $2$, we have
\begin{align*}
\Ass(R/I)=& \{ \frk{p}_1=( x_1, x_2, x_3, x_4), \frk{p}_2=( x_2, x_3, x_4, y_1, y_2), \frk{p}_3=(  x_3,  x_4, y_1,y_2, y_3),\\
& \frk{p}_4=( y_1, y_2, y_3, y_4) \}.
\end{align*}
By considering the sequence $\frk{p}_1, \frk{p}_2, \frk{p}_3, \frk{p}_4$ of minimal prime ideals of $I(G)$, we have $\hte(\frk{p}_i +\frk{p}_{i+1})\leq 6= \hte(I)+2$ for $1 \leq i \leq 3$. Therefore for any pair of distinct minimal prime ideals $\frk{p}$ and $\frk{q}$ in $I(G)$, there exists a  sequence of minimal prime ideals $\frk{p}=\frk{p}_1, \ldots, \frk{p}_r=\frk{q}$ such that $\hte(\frk{p}_i +\frk{p}_{i+1})\leq 6= \hte(I)+2$. Therefore $I(G)$ is connected in codimension $2$. But $I(G)$ is not unmixed.
\end{Example}

\begin{Theorem}\label{T12}
Let $G$ be an unmixed aCM bipartite graph with vertex partition $V_1 \cup V_2$. Then the vertices $V_1 = \{ x_1, \ldots, x_n \}$ and $V_2= \{ y_1, \ldots, y_n \}$ can be labeled such that:
\begin{enumerate}
\item[(i)] $\{ x_i, y_i \}$ are edges for $i = 1, \ldots, n$;
\item[(ii)] if $\{ x_i, y_j \}$ is an edge, then $i \leq j+1$;
\item[(iii)] if $\{ x_i, y_j \}$ and $ \{ x_j , y_k \} $ are edges, then $\{x_i, y_k \}$ is an edge.
\end{enumerate}
\end{Theorem}
\begin{proof}
Since $G$ is unmixed bipartite graph, by Theorem \ref{T1} we have $\{ x_i, y_i \}$ are edges for $i = 1, \ldots, n$ and if $\{ x_i, y_j \}$ and $ \{ x_j , y_k \} $ are edges, then $\{x_i, y_k \}$ is an edge of $G$. Let $\Delta$ be the simplicial complex with $I_{\Delta} = I(G)$. Then $V_1$ and $V_2$ are facets of
$\Delta$, while $\{ x_i, y_i \}$ are not faces of $\Delta$.
By Theorem \ref{T2}, $\Delta$ is connected in codimension two, it follows that there is a
sequence of facets $F_1, \ldots , F_s$ of $\Delta$ with $F_1 = V_1$ and $F_s = V_2$ such that
$|F_{k-1} \cap F_k| \geq n -2$ for $k = 1, \ldots, s$. Then $|F_2 \setminus F_1 | \leq 2$,
say $F_2 \setminus F_1 = \{ y_1 \}$ or $F_2 \setminus F_1 = \{ y_1, y_2 \}$. Therefore $F_2 = \{ y_1, x_2, \ldots, x_n \}$ or $F_2 = \{ y_1,y_2, x_3, \ldots, x_n \}$ because $\{ x_1, y_1 \}$ and $\{ x_2, y_2 \}$ are not faces and $\Delta$ is pure. A similar argument as above implies that $|F_k \setminus F_{k-1} | \leq 2$ for $k=1, \cdots, s$ and hence we may assume that $F_k = \{ y_1, \ldots, y_i, x_{i+1}, \ldots, x_n \}$ for $k = 1, \ldots, s$ and for some $i$ such that if $ i > j+1$, then $\{x_i, y_j\}$ is a face of $\Delta$. Therefore, if $ i > j+1$, then $\{x_i, y_j\}$  is not an edge of $G$. In other word, if $\{x_i, y_j \}$ is an edge of $G$, then $i \leq j+1$.

\end{proof}

The following example shows that the converse of Theorem \ref{T12} is not true.

\begin{Example}
Let $G$ be the following graph which is unmixed and connected in codimension two and also  if $\{ x_i, y_j \} \in E(G)$, then $i \leq j+1$. But $G$ is not aCM, since $\dim(R/I)=6$ and $\depth(R/I)=4$.

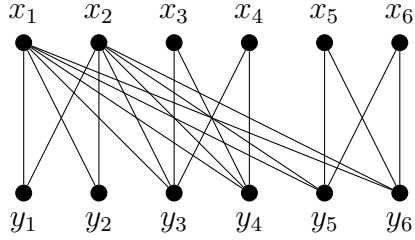
\[\begin{tikzpicture}
	\vertex[fill] (x1) at (0,2) [label=above:$x_{1}$] {};
	\vertex[fill] (x2) at (1,2) [label=above:$x_{2}$] {};
	\vertex[fill] (x3) at (2,2) [label=above:$x_{3}$] {};
\vertex[fill] (x4) at (3,2) [label=above:$x_{4}$] {};
\vertex[fill] (x5) at (4,2) [label=above:$x_{5}$] {};
\vertex[fill] (x6) at (5,2) [label=above:$x_{6}$] {};
	\vertex[fill] (y1) at (0,0) [label=below:$y_{1}$] {};
	\vertex[fill] (y2) at (1,0) [label=below:$y_{2}$] {};
	\vertex[fill] (y3) at (2,0) [label=below:$y_{3}$] {};
\vertex[fill] (y4) at (3,0) [label=below:$y_{4}$] {};
\vertex[fill] (y5) at (4,0) [label=below:$y_{5}$] {};
\vertex[fill] (y6) at (5,0) [label=below:$y_{6}$] {};
	\path
		(x1) edge (y1)
		(x1) edge (y2)
		(x1) edge (y3)
		(x1) edge (y4)
		(x1) edge (y5)
		(x1) edge (y6)
		(x2) edge (y1)
		(x2) edge (y2)
		(x2) edge (y3)
(x2) edge (y4)
(x2) edge (y5)
(x2) edge (y6)
		(x3) edge (y3)
		(x3) edge (y4)
		(x4) edge (y3)
		(x4) edge (y4)
		(x5) edge (y5)
		(x5) edge (y6)
		(x6) edge (y5)
		(x6) edge (y6)
		
	;
\end{tikzpicture}\]
By Macaulay $2$, we have
\begin{align*}
\Ass(R/I)=& \{ ( x_1, x_2, x_3, x_4, x_5, x_6), ( x_1, x_2, x_3, x_4, y_5, y_6), ( x_1, x_2, x_5, x_6, y_3,y_4),\\
& ( x_1, x_2, y_3, y_4, y_5, y_6), ( y_1, y_2, y_3, y_4, y_5, y_6) \}.
\end{align*}
It is easy to check that for any pair of distinct minimal prime ideals $\frk{p}$ and $\frk{q}$ in $I(G)$, there exists a  sequence of minimal prime ideals $\frk{p}=\frk{p}_1, \ldots, \frk{p}_r=\frk{q}$ such that $\hte(\frk{p}_i +\frk{p}_{i+1})\leq 8= \hte(I)+2$. Therefore $I(G)$ is connected in codimension two.
\end{Example}

Given a positive integer $n$, and an $n$-tuple $\lambda=(\lambda_{1}, \ldots, \lambda_n)$ of positive integers such that
$m= \lambda_1 \geq \lambda_2 \geq \ldots \geq \lambda_n $, the Ferrers graph $G=G_{\lambda}$ associated to $\lambda$ defined by Corso and Nagel \cite{CN} and it is the bipartite graph with bipartition $V_1=\{ x_1, \ldots, x_n \}$ and $V_2=\{ y_1, \ldots, y_m \}$ such that if $(x_i, y_j)$ is an edge of $G$, then so is $(x_r, y_s)$ for $1 \leq r \leq i$ and $1 \leq s \leq j$. Suppose that $x_1, \ldots , x_n, y_1, \ldots , y_m$ are indeterminates over the field $k$. The
edge ideal of $G$ in the polynomial ring $R = k[x_1, \ldots , x_n, y_1, \ldots , y_m]$ denoted by $I_{\lambda}$.

In \cite{CN}, it is shown that
\begin{equation}\label{E.F1}
I_{\lambda} =\cap_{i=1}^{n+1} (x_1, \ldots, x_{i-1}, y_1, \ldots, y_{\lambda_{i}})
\end{equation}
where, by convenience of notation, we have set $\lambda_{n+1} =0$. Set $c_1 = 1$, and suppose that
\[ \lambda_{1}= \ldots = \lambda_{c_{2}-1} > \lambda_{c_2} = \ldots= \lambda_{c_{3}-1} > \lambda_{c_3}= \ldots= \lambda_{c_{k}-1} > \lambda_{c_{k}}=\ldots= \lambda_{n} . \]
Finally set $c_{k+1} = n + 1$. Then a minimal prime decomposition of $I_{\lambda}$ can be
obtained as follows, by omitting redundant terms from (\ref{E.F1}):
\begin{equation}\label{E.F2}
I_{\lambda} =\cap_{i=1}^{k+1} (x_1, \ldots, x_{c_{i}-1}, y_1, \ldots, y_{\lambda_{c_i}})
\end{equation}

\begin{Theorem}
Let $G$ be a unmixed Ferrers graph with associated partition $\lambda = (\lambda_1, \lambda_2, \ldots, \lambda_n)$ and let $I_{\lambda}$
be the edge ideal in $R=k[x_1, \ldots , x_n, y_1, \ldots , y_n]$ associated to $G$. Then $I_{\lambda}$ is aCM if and only if  $I_{\lambda}$ is connected in codimension two.

\end{Theorem}\label{T5}
\begin{proof}
$(\Longrightarrow). $ This is obvious by Theorem \ref{T2}.\\
$(\Longleftarrow).$ Let $V_1=\{ x_1, \ldots, x_n \}$ and $V_2=\{ y_1, \ldots, y_n \}$. Since $I_{\lambda}$ is the Ferrers ideal, by \cite[Corollary 2.2]{CN}, we have $\hte(I)=\min \{ \min_{1\leq j\leq n} \{ \lambda_{j}+j-1 \}, n \}$ and $\pd(R/I)=\max_{1\leq j\leq n} \{ \lambda_{j}+j-1 \}$. By \cite[Proposition 2.3]{MN2} it is enough to show that $n \leq \lambda_{j}+j-1 \leq n+1 $ for $j=1, \ldots, n$.
By \cite[Corollary 2.6]{CN}, we have
\begin{align*}
\Ass(R/I)= \{ (y_1, \ldots, y_n), (x_1, \ldots, x_{c_{2}-1}, y_1, \ldots, y_{\lambda_{c_2}}), \ldots, (x_1, \ldots, x_n) \}
\end{align*}
Since $I_{\lambda}$ is unmixed we have, $\lambda_{c_i}+c_{i}-1 =n$. Also since $I_{\lambda}$ is connected in codimension two, we have $\lambda_{c_i}-\lambda_{c_{i-1}} \leq 2$ and $c_{i} - c_{i-1} \leq 2$. Therefore for any $j$, there exists $\lambda_{c_{k}}$ such that $\lambda_{c_{k}}=\lambda_{j}$ and $ c_k\leq j \leq c_{k}+1$. Therefore $\max_{1\leq j\leq n} \{ \lambda_{j}+j-1 \} \leq n+1 $. Hence $\hte(I) +1 \geq \pd (R/I)$.
\end{proof}

--------------------------------------------------------------------------------

\end{document}